\documentclass[conference]{IEEEtran}
\IEEEoverridecommandlockouts
\usepackage{color}
\usepackage{theoremref, hyperref}
\usepackage{amsmath, amsthm}
\usepackage{amssymb}

\numberwithin{equation}{section}

\definecolor{darkcyan}{rgb}{0.0, 0.55, 0.55}

\newcommand{\Hil}[0]{
\mathcal{H}
}

\newcommand{\Rd}[0]{
{\mathbb{R}^d}
}

\newcommand{\B}[0]{{\mathcal{B}}}

\newcommand{\A}[0]{{\mathcal{A}}}
\newcommand{\J}[0]{{\mathcal{J}}}

\newtheorem{theorem}{Theorem}[section]
\newtheorem{definition}[theorem]{Definition}
\newtheorem{proposition}[theorem]{Proposition}
\newtheorem{lemma}[theorem]{Lemma}
\newtheorem{corollary}[theorem]{Corollary}

\newtheorem{conj.}[theorem]{Conjecture}
\newtheorem{Bsp.}[theorem]{Example}

\usepackage{cite}
\usepackage{amsmath,amssymb,amsfonts}
\usepackage{algorithmic}
\usepackage{graphicx}
\usepackage{textcomp}
\usepackage{xcolor}
\def\BibTeX{{\rm B\kern-.05em{\sc i\kern-.025em b}\kern-.08em
    T\kern-.1667em\lower.7ex\hbox{E}\kern-.125emX}}
\begin{document}

\title{On the inverse-closedness of operator-valued matrices with polynomial off-diagonal decay}

\author{\IEEEauthorblockN{1\textsuperscript{st} Lukas Köhldorfer}
\IEEEauthorblockA{\textit{Acoustics Research Institute} \\
\textit{Austrian Academy of Sciences}\\
Vienna, Austria \\
lukas.koehldorfer@oeaw.ac.at}
\and
\IEEEauthorblockN{2\textsuperscript{nd} Peter Balazs}
\IEEEauthorblockA{\textit{Acoustics Research Institute} \\
\textit{Austrian Academy of Sciences}\\
Vienna, Austria \\
peter.balazs@oeaw.ac.at}
}

\maketitle

\begin{abstract}
We give a self-contained proof of a recently established $\B(\Hil)$-valued version of Jaffards Lemma. That is, we show that the Jaffard algebra of $\B(\Hil)$-valued matrices, whose operator norms of their respective entries decay polynomially off the diagonal, is a Banach algebra which is inverse-closed in the Banach algebra $\B(\ell^2(X;\Hil))$ of all bounded linear operators on $\ell^2(X;\Hil)$, the Bochner-space of square-summable $\Hil$-valued sequences.
\end{abstract}


\section{Introduction}
Let $\Hil$ be a Hilbert space and $\ell^2(X;\Hil)$ be the Bochner-space of $\Hil$-valued sequences $(g_l)_{l\in X}$, whose associated sequence of norms $(\Vert g_l \Vert)_{l\in X}$ is square-summable. Then $\ell^2(X;\Hil)$ itself is a Hilbert space and $\B(\ell^2(X;\Hil))$, the space of all bounded linear operators on $\ell^2(X;\Hil)$, is a C*-algebra. By the \emph{matrix calculus} for $\B(\ell^2(X;\Hil))$ as explained in \cite[Sec. 3.1]{koebacasheihomosha23}, every operator $A\in \B(\ell^2(X;\Hil))$ can be uniquely identified with a $\B(\Hil)$-valued matrix $[A_{k,l}]_{k,l\in X}$ in the sense that the action of $A$ on $(g_l)_{l\in X}\in \ell^2(X;\Hil)$ precisely corresponds to matrix multiplication of $[A_{k,l}]_{k,l\in X}$ with $(g_l)_{l\in X}$ viewed as column vector. Moreover, composition of operators from $\B(\ell^2(X;\Hil))$ corresponds to matrix multiplication of their respective matrix representations, and taking adjoints in $\B(\ell^2(X;\Hil))$ corresponds to the involution
\begin{equation}\label{involution}
([A_{k,l}]_{k,l\in X})^* = ([A_{k,l}^*]_{k,l\in X})^T,    
\end{equation}
where the exponent $T$ denotes transposition. Thus $\B(\ell^2(X;\Hil))$ can be identified with a certain Banach algebra of $\B(\Hil)$-valued matrices (see also \cite[Thm. 5.28]{kohl21}). 

Based on the latter point of view, we consider the \emph{Jaffard class} $\J_s = \J_s(X;\B(\Hil))$ of $\B(\Hil)$-valued matrices $A=[A_{k,l}]_{k,l\in X}$, for which there exists $C>0$, such that
\begin{equation}
    \Vert A_{k,l} \Vert \leq C (1+\vert k-l\vert)^{-s} \qquad (\forall k,l\in X).  
\end{equation}
It has been shown \cite{koeba25}, that for sufficiently regular index sets $X$ and a sufficiently large decay parameter $s>0$, the Jaffard class is a unital Banach *-algebra with respect to matrix multiplication and involution as in (\ref{involution}), which is contained in $\B(\ell^2(X;\Hil))$ and, in fact, \emph{inverse-closed} in $\B(\ell^2(X;\Hil))$, meaning that  
\begin{equation}\label{Wienerpair}
A \in \J_s \text{ and } \exists A^{-1} \in \B(\ell^2(X;\Hil)) \, \Longrightarrow \, A^{-1} \in \J_s.    
\end{equation}
Property (\ref{Wienerpair}) can be seen as a variant of \emph{Wiener's Lemma} on absolutely convergent Fourier series \cite{Wiener32} for $\B(\Hil)$-valued matrices with polynomial off-diagonal decay and is known as \emph{Jaffard's Lemma} \cite{ja90} in the scalar-valued setting. This result (compare also with \cite{koeba25,Baskakov1990,Bas97,baskri14,zbMATH06949773,kri11}) is not only interesting from an abstract point of view, but also has an enormous potential as a powerful tool for the study of localized g-frames \cite{koeba25-2,Krishtal2011}, operator theory \cite{douglas}, the study of Fourier series of operators \cite{feiko98} or harmonic quantum analysis \cite{werner84,Gosson21}.

In \cite{koeba25} the inverse-closedness of $\J_s$ in $\B(\ell^2(X;\Hil))$ is deduced from the inverse-closedness of certain weighted Schur-type algebras in $\B(\ell^2(X;\Hil))$ via methods from \cite{groelei06}. Here we give a more self-contained presentation of this fact by adapting methods from \cite{sun07a,hol22} to the $\B(\Hil)$-valued setting.

\section{Results}

\subsection{Banach algebra properties}

We consider $\B(\Hil)$-valued matrices indexed by a \emph{relatively separated} set $X\subset \Rd$, meaning that 
\begin{equation}\label{relativelyseparated}
\sup_{x\in \Rd} \vert X \cap (x+[0,1]^d) \vert < \infty.
\end{equation}
For such $X$ the following hold: 
\begin{lemma}\label{separatedlemma}
    Let $X\subset \Rd$ be a relatively separated set. 
\begin{itemize}
    \item[(a)] \cite[Lemma 1]{gr04-1} For any $s>d$, there exists a constant $C=C(s)>0$ such that 
    $$\sup_{x\in \Rd} \sum_{k\in X} (1+\vert x-k\vert)^{-s} = C <\infty$$
    \item[(b)] \cite[Lemma 2 (a)]{gr04-1} For any $s>d$, there exists a constant $C=C(s)>0$ such that for all $k,l\in X$
    $$\sum_{n\in X} (1+\vert k-n\vert)^{-s}(1+\vert l-n\vert)^{-s} \leq C (1+\vert k-l \vert)^{-s}.$$
\end{itemize}
\end{lemma}

\begin{definition}
Let $X\subset \Rd$ be relatively separated and $\nu_s (x):\Rd\longrightarrow [0,\infty)$ the polynomial weight function given by $\nu_s (x) = (1+\vert x \vert)^s$, $s\geq 0$. Let $\mathcal{J}_s = \mathcal{J}_s(X;\B(\Hil))$ be the space of all $\mathcal{B}(\Hil)$-valued matrices $A=[A_{k,l}]_{k,l\in X}$, for which 
\begin{equation}
    \Vert A \Vert_{\mathcal{J}_s} := \sup_{k,l\in X} \Vert A_{k,l} \Vert \nu_s(k-l)  
\end{equation}
is finite. We call $\mathcal{J}_s$ the \emph{Jaffard class} (and later $-$once justified$-$ the \emph{Jaffard algebra}) and $\Vert \, . \, \Vert_{\mathcal{J}_s}$ the \emph{Jaffard norm}. 
\end{definition}

Since $\mathcal{J}_s$ is nothing else than the weighted Bochner space $\mathcal{J}_s = \ell_{u_s}^{\infty}(X\times X; \B(\Hil))$, where $u_s(k,l) = \nu_s(k-l)$, the Jaffard class $(\mathcal{J}_s, \Vert\, . \, \Vert_{\mathcal{J}_s})$ is a Banach space \cite[Chap. 1]{HyNeVeWe16}.

\begin{proposition}\label{propboundedonKp}
Let $s>d+r$, $r\geq0$ and $m:\mathbb{R}^d\rightarrow [0,\infty)$ be a function for which there exists some $C>0$, such that $m(x+y)\leq C m(y) \nu_r(y)$ for all $x,y\in \Rd$. Then each $A \in \mathcal{J}_s$ defines a bounded operator on $\ell^p_m(X;\Hil)$ for every $\frac{d}{s-r} < p \leq \infty$.
\end{proposition}

\begin{proof}
For $\frac{d}{s-r} < p \leq 1$, the function $f:\mathbb{R}_{\geq 0}\longrightarrow \mathbb{R}_{\geq 0}, f(x)=x^p$ is subadditive. Hence, for $g=(g_l)_{l\in X} \in \ell^p_m(X;\Hil)$, we obtain
\begin{flalign}
&\Vert Ag\Vert_{\ell^p_m(X;\Hil)}^p \notag \\
&= \sum_{k\in X} \left\Vert \sum_{l\in X} A_{k,l} g_l \right\Vert^p m(k)^p \notag \\
&\leq C^p \sum_{k\in X} \sum_{l\in X} \Vert A_{k,l}\Vert^p \Vert g_l\Vert^p m(l)^p (1+\vert k-l\vert)^{rp}  \notag \\
&\leq C^p \Vert A \Vert_{\J_s}^p \sum_{l\in X} \Vert g_l\Vert^p m(l)^p \sum_{k\in X} (1+\vert k-l\vert)^{-p(s-r)}  \notag \\
&\leq C_1 \Vert A \Vert_{\J_s}^p \Vert g\Vert_{\ell^p_m(X;\Hil)}^p, \notag 
\end{flalign}
where we applied \cite[Lemma 2 (a)]{gr04-1} is the last step. 

The case $p=\infty$ is proven similarly.
%

Finally, the case $1<p<\infty$ follows from Riesz-Thorin interpolation \cite[Thm. 2.2.1]{HyNeVeWe16}.
\end{proof}

Setting $r=0$ and $m\equiv 1$ above yields the following. 

\begin{corollary}\label{corboundedKp}
If $s>d$, then 
$$\mathcal{J}_s(X) \subset \bigcap_{1\leq p \leq \infty} \B(\ell^p(X;\Hil)).$$
\end{corollary}

Since $\J_s \subset \B(\ell^2(X;\Hil))$, the matrix calculus \cite[Sec. 3.1]{koebacasheihomosha23} for the Banach algebra $\B(\ell^2(X;\Hil))$, mentioned in the introduction, motivates us to define a multiplication on $\J_s$ via matrix multiplication. For the same reason, we define an involution on $\J_s$ as in (\ref{involution}).

\begin{proposition}\label{lemmaalgebra}
For $s>d$, the Jaffard class $(\mathcal{J}_s, \Vert \, . \, \Vert_{\mathcal{J}_s})$ is a unital *-algebra with respect to matrix multiplication and involution as defined in (\ref{involution}). Furthermore, the involution is an isometry.
\end{proposition}

\begin{proof}
Let $A=[A_{k,l}]_{k,l\in X}, B=[B_{k,l}]_{k,l\in X} \in \J_s$. For arbitrary $k,l\in X$ we have
\begin{flalign}
\Vert [A\cdot B]_{k,l} \Vert &\leq \sum_{n\in X} \Vert A_{k,n}\Vert \, \Vert B_{n,l} \Vert \notag \\
&\leq \Vert A \Vert_{\J_s} \, \Vert B \Vert_{\J_s} \sum_{n\in X} \nu_s(k-n)^{-1} \nu_s(n-l)^{-1} \notag \\
&\leq C \Vert A \Vert_{\J_s} \, \Vert B \Vert_{\J_s} \nu_s(k-l)^{-1} ,\notag
\end{flalign}
where we applied Lemma \ref{separatedlemma} (b) in the last step. Consequently
\begin{equation}\label{notBanachalgebranorm}
    \Vert A\cdot B\Vert_{\J_s} \leq C \Vert A \Vert_{\J_s} \, \Vert B \Vert_{\J_s}, 
\end{equation}
which implies that the Jaffard class is an algebra. Moreover, the neutral element $\mathcal{I}_{\J_s} = \mathcal{I}_{\B(\ell^2(X;\Hil))} = \text{diag}[\mathcal{I}_{\B(\Hil)}]_{k\in X}$ is contained in $\J_s$, since $\Vert \mathcal{I}_{\J_s} \Vert_{\J_s} = \nu_s(0) = 1$. The involution property in $\J_s$ follows from the matrix calculus with respect to $\B(\ell^2(X;\Hil))$ via Corollary \ref{corboundedKp}. Finally, $\nu_s(-x) = \nu_s(x)$ for all $x\in \Rd$ implies that the involution is an isometry.
\end{proof}

By (\ref{notBanachalgebranorm}), the Jaffard norm is \emph{not} a Banach algebra norm. However, we may equip the Jaffard class with the equivalent norm 
\begin{equation}\label{Jaffardnorm}
    \vert\Vert A \Vert\vert_{\mathcal{J}_s} := \sup_{\substack{B\in \J_s \\ \Vert B \Vert_{\mathcal{J}_s} = 1}} \Vert A\cdot B \Vert_{\mathcal{J}_s} \qquad (A \in \mathcal{J}_s),
\end{equation}
which is indeed a Banach algebra norm. As a consequence we obtain:

\begin{corollary}\label{propJaffardalgebra}\cite{koeba25}
For any $s>d$, $(\J_s, \vert\Vert \, . \, \Vert\vert_{\mathcal{J}_s})$ is a unital Banach *-algebra.    
\end{corollary}

\subsection{Inverse-closedness}

In this section we show property (\ref{Wienerpair}), i.e., that $\J_s$ is inverse-closed in $\B(\ell^2(X;\Hil))$. While inverse-closedness in principal is an algebraic property, \emph{Hulanicki's Lemma} \cite{Hulanicki1972} allows for an analytical treatment of this task. We refer the reader to \cite[Prop. 2.5]{gr10-2} for a proof of this simple but beautiful fact. Let $\sigma_{\B}(B)$ and $r_{\B}(B)$ denote the spectrum and the spectral radius of an element $B$ from a Banach algebra $\B$, respectively.

\begin{proposition}[Hulanicki's Lemma]\label{Hulanickilemma}\cite{Hulanicki1972}
Let $\mathcal{A}\subseteq \mathcal{B}$ be a pair of unital Banach *-algebras with common identity and common involution, and suppose that $\mathcal{B}$ is symmetric, i.e. $\sigma_{\B}(B^*B) \subseteq [0,\infty)$ ($\forall B\in \B$). Then the following are equivalent:
\begin{itemize}
\item[(i)] $\mathcal{A}$ is inverse-closed in $\mathcal{B}$.
\item[(ii)] $r_{\mathcal{A}}(A)=r_{\mathcal{B}}(A) \qquad (\forall A=A^* \in \mathcal{A})$.
\item[(iii)] $r_{\mathcal{A}}(A) \leq r_{\mathcal{B}}(A) \qquad (\forall A=A^* \in \mathcal{A})$.
\end{itemize}
In case the above conditions hold, $\A$ is symmetric as well. 
\end{proposition}

In order to be able to employ Hulanicki's Lemma, we need some preparation. 

\begin{lemma}\label{jakoblemma}\cite[Lemma 5.13]{hol22}
Let $X \subset \Rd$ be relatively separated, $s>d$ and $\tau_0 >0$ be given. For any $k\in X$ and any $\tau > \tau_0$, let $M_{1,k}^{\tau} := \lbrace n\in X: \vert k-n \vert_{\infty} \leq \lceil \tau \rceil \rbrace$ and $M_{2,k}^{\tau} := \lbrace n\in X: \vert k-n \vert_{\infty} > \lceil \tau \rceil \rbrace$. Then there exists a constant $C=C(X,s,\tau_0)>0$, such that for all $k\in X$ and all $\tau > \tau_0$
\begin{equation}\label{tau}
    \vert M_{1,k}^{\tau} \vert \leq C\tau^d \quad \text{and} \quad \sum_{n\in M_{2,k}^{\tau}} \nu_s(k-n)^{-1} \leq C\tau^{d-s} .
\end{equation}
\end{lemma}

\begin{proof} In order to show the first inequality in (\ref{tau}), recall that $X \subset \Rd$ being relatively separated means that $\gamma := \sup_{x\in \Rd} \vert X \cap (x+[0,1]^d) \vert$ is finite. Since $M_{1,k}^{\tau}$ can be covered by $(2 \lceil \tau \rceil)^d$ many translated unit cubes, we see that $\vert M_{1,k}^{\tau} \vert \leq (2 \lceil \tau \rceil)^d \gamma$. Since $\lceil \tau \rceil \leq 1+\tau \leq (\frac{1}{\tau_0} +1)\tau$, we obtain $\vert M_{1,k}^{\tau} \vert \leq 2^d \gamma (\frac{1}{\tau_0} +1)^d \tau^d$.

In order to show the second inequality in (\ref{tau}) fix some arbitrary $k\in X$. Since $\vert z \vert_{\infty} \leq \vert z \vert$ for all $z\in \Rd$, it suffices to estimate the series $\sum_{n\in M_{2,k}^{\tau}} (1+ \vert k-n\vert_{\infty} )^{-s}$. We may assume W.L.O.G. that $X$ is separated, i.e. that $\inf_{x,y\in X, x\neq y} \vert x-y \vert =: \delta > 0$, since any relatively separated set $X\subset \Rd$ is a finite union of separated sets \cite[Sec. 9.1]{ole1n}. Now, observe that if $n\in M_{2,k}^{\tau}$, i.e. $n\in X$ and $\vert k-n\vert_{\infty} > \lceil \tau \rceil$, then $\exists l\in S^{\tau} := \lbrace l\in \mathbb{Z}^d : \vert l \vert_{\infty} \geq \lceil \tau \rceil \rbrace$ such that $n = k+l+x$ for some $x\in [0,1)^d$. Thus, if for any $l\in S^{\tau}$ we define $X_{l,k} := X\cap (k+l+[0,1)^d)$, then we see that the family $\lbrace X_{l,k} : l\in S^{\tau}\rbrace$ covers $M_{2,k}^{\tau}$, hence 
$$\sum_{n\in M_{2,k}^{\tau}} (1+ \vert k-n\vert_{\infty} )^{-s} \leq \sum_{l\in S^{\tau}}\sum_{n\in X_{l,k}} (1+ \vert k-n\vert_{\infty} )^{-s} = (\ast).$$
Since, by our previous observation, $n = k+l+x$ for some $l\in S^{\tau}$ and some $x\in [0,1)^d$, we see that $1+ \vert k-n\vert_{\infty} = 1+ \vert l+x\vert_{\infty} \geq 1 +\vert l \vert_{\infty} - \vert x\vert_{\infty} \geq \vert l \vert_{\infty}$. Using this together with the observation that $\vert X_{l,k}\vert \leq C'$, where the constant $C'$ only depends on $\delta$ and the dimension $d$, we obtain that 
$$(\ast) \leq C' \sum_{l\in S^{\tau}} \vert l \vert_{\infty}^{-s} =: (\ast \ast).$$
Now, for each $m\in \mathbb{N}$ with $m \geq \lceil \tau \rceil$ set $S_m := \lbrace l\in \mathbb{Z}^d: \vert l \vert_{\infty} = m \rbrace$. Then $S^{\tau} = \bigcup_{m \geq \lceil \tau \rceil} S_m$ and each $S_m$ consists of the integer lattice points located on the surface of a cube of side-length $2m$. Since the number of lattice points on each face of such a cube equals $(2m+1)^{d-1}$ and there are $2d$ faces per cube in total, we obtain
\begin{flalign}
(\ast \ast) &\leq C' \sum_{m=\lceil \tau \rceil}^{\infty} \sum_{l\in S_m} \vert l \vert_{\infty}^{-s} \notag \\
&\leq 2d C' \sum_{m=\lceil \tau \rceil}^{\infty} m^{-s} (2m+1)^{d-1} \notag \\
&\leq 2d C' \sum_{m=\lceil \tau \rceil}^{\infty} m^{-s} (4m)^{d-1} \notag \\
&= 2d C' 4^{d-1} \left(\lceil \tau \rceil^{d-s-1} +   \sum_{m=\lceil \tau \rceil +1}^{\infty} m^{d-s-1} \right). \notag
\end{flalign}
Since $\lceil \tau \rceil^{d-s-1} \leq \lceil \tau \rceil^{d-s}$ and $m^{d-s-1} \leq \int_{m-1}^m x^{d-s-1} \, dx$, we obtain in total that  
\begin{flalign}
\sum_{n\in M_{2,k}^{\tau}} (1+ \vert k-n\vert )^{-s} 
& \leq C'' \left(\lceil \tau \rceil^{d-s} + \int_{\lceil \tau \rceil}^{\infty} x^{d-s-1} \, dx \right) \notag \\
&= C''\left(1+ \frac{1}{s-d}\right) \lceil \tau \rceil^{d-s}, \notag    
\end{flalign}
which yields the desired inequality.
\end{proof}

Next we show that $\mathcal{J}_s$ is continuously embedded in $\mathcal{B}(\ell^2(X;\Hil))$.

\begin{lemma}\label{propembedding}
Let $s>d$. Then there exists $C>0$, such that 
    \begin{equation}\label{embedding}
        \Vert A \Vert_{\mathcal{B}(\ell^2(X;\Hil))} \leq C \Vert A \Vert_{\mathcal{J}_s} \qquad (\forall A \in \mathcal{J}_s).
    \end{equation}
\end{lemma}

\begin{proof}
The statement is proved essentially as \cite[Lemma 5.2]{hol22}:

Let $A\in \J_s$ be arbitrary. Since $(\J_s, \vert\Vert \, . \, \Vert\vert_{\mathcal{J}_s})$ is a unital Banach *-algebra by Corollary \ref{propJaffardalgebra}, we have by Gelfand's formula for the spectral radius that $r_{\J_s}(A) =  \lim_{n\rightarrow \infty} \vert\Vert A^n \Vert\vert_{\J_s}^{\frac{1}{n}}$ for any $A\in \J_s$. Furthermore, since $\J_s \subset \mathcal{B}(\ell^2(X;\Hil))$ by Corollary \ref{corboundedKp}, a simple argument (see e.g. \cite[Lemma 2.4]{gr10-2}) shows that $ r_{\mathcal{B}(\ell^2(X;\Hil))}(A) \leq r_{\J_s}(A)$ for all $A\in \J_s$. Thus  
\begin{flalign}
\Vert A \Vert_{\mathcal{B}(\ell^2(X;\Hil))}^2 &= \Vert A^* A \Vert_{\mathcal{B}(\ell^2(X;\Hil))} \notag \\
&= r_{\mathcal{B}(\ell^2(X;\Hil))}(A^* A) \notag \\
&\leq r_{\J_s}(A^* A) \notag \\
&= \lim_{n\rightarrow \infty} \vert\Vert (A^* A)^n \Vert\vert_{\J_s}^{\frac{1}{n}} \notag \\
&\leq \vert\Vert A^* A \Vert\vert_{\J_s} \notag \\
&\leq \vert\Vert A^* \Vert\vert_{\J_s} \vert\Vert A \vert\Vert_{\J_s} \notag \\
&\leq C \Vert A \Vert_{\J_s}^2 , \notag
\end{flalign}
for all $A\in \J_s$, where we used (\ref{Jaffardnorm}) in the last line.
\end{proof}

\begin{lemma}\label{opnormineq}
Let $A = [A_{k,l}]_{k,l\in X}$ be a $\mathcal{B}(\Hil)$-valued matrix and $1\leq p \leq \infty$. If $A$ defines an element in $\mathcal{B}(\ell^p(X;\Hil))$, then the following hold:
\begin{itemize}
    \item[(a)] For $1\leq p < \infty$, 
    $$\sup_{l\in X} \sup_{\Vert f \Vert_{\Hil} = 1} \left( \sum_{k\in X} \Vert A_{k,l} f \Vert_{\Hil}^p \right)^{\frac{1}{p}} \leq \Vert A \Vert_{\mathcal{B}(\ell^p(X;\Hil))}.$$
    \item[(b)] $$\sup_{k,l \in X} \Vert A_{k,l} \Vert \leq \Vert A \Vert_{\mathcal{B}(\ell^p(X;\Hil))}.$$ 
\end{itemize}
\end{lemma}

\begin{proof}
(a) For $l\in X$ and $1\leq p \leq \infty$, define $P_l^p:\ell^p(X;\Hil) \longrightarrow \ell^p(X;\Hil)$, $P_l^p (f_k)_{k\in X} = (\delta_{k,l} f_k)_{k\in X}$. Then for arbitrary but fixed $l\in X$ and $1\leq p < \infty$ we see that
\begin{flalign}
    \Vert A \Vert_{\B(\ell^p(X;\Hil))}^p &= \sup_{\Vert f\Vert_{\ell^p(X;\Hil)} =1} \sum_{k\in X} \left\Vert \sum_{l\in X} A_{k,l} f_l  \right\Vert^p \notag \\
    &\geq \sup_{\substack{ f = P_l^p f \\ \Vert f\Vert_{\ell^p(X;\Hil)} =1}} \sum_{k\in X} \Vert A_{k,l} f_l \Vert^p \notag \\
    &= \sup_{\Vert f\Vert_{\Hil}=1} \sum_{k\in X} \Vert A_{k,l} f \Vert^p . \notag
\end{flalign}
Taking the $p$-th root and supremum over all $l\in X$ yields (a). 

(b) For $1\leq p < \infty$, we have by (a) that 
$$\Vert A_{k,l}\Vert \leq \sup_{\Vert f\Vert_{\Hil}=1} \left( \sum_{k\in X} \Vert A_{k,l} f \Vert^p \right)^{\frac{1}{p}} \leq \Vert A \Vert_{\mathcal{B}(\ell^p(X;\Hil))}$$ 
for all $k,l\in X$. Taking the supremum over all $k,l\in X$ yields the claim. The case $p=\infty$ is omitted. 
\end{proof}

After these preparatory results, we are able to prove the decisive ingredient for proving our main theorem. 

\begin{lemma}\label{lemmagamma}
Let $s>d$ and $\gamma = 1-\frac{d}{s} > 0$. Then there exists a positive constant $C$, such that 
    \begin{equation}\label{gammaineq}
        \Vert A^2 \Vert_{\mathcal{J}_s} \leq C \Vert A \Vert_{\mathcal{J}_s}^{2-\gamma} \Vert A \Vert_{\mathcal{B}(\ell^2(X;\Hil))}^{\gamma} \qquad (\forall A \in \mathcal{J}_s).
    \end{equation}
\end{lemma}

\begin{proof}
We abbreviate $\mathcal{B} = \mathcal{B}(\ell^2(X;\Hil))$. Let $A\in \J_s$ be arbitrary and assume W.L.O.G. that $A\neq 0$. Hölder's inequality on $\mathbb{R}^2$ with respect to the exponent $s$ (and its conjugated Hölder exponent) implies that
$$\nu_s(k-l) < 2^{s} \left(\nu_s(k-n) + \nu_s(n-l)\right) \qquad (\forall k,l,n \in X).$$
Thus, for arbitrary $k,l \in X$ we can estimate
\begin{flalign}\label{est1}
&\Vert [A^2]_{k,l} \Vert \nu_s(k-l) \notag \\
&\leq \sum_{n\in X} \Vert A_{k,n} \Vert \Vert A_{n,l} \Vert \nu_s(k-l) \notag \\
&< 2^{s} \sum_{n\in X} \Vert A_{k,n} \Vert \Vert A_{n,l} \Vert \left( \nu_s(k-n) + \nu_s(n-l)\right) \notag \\
&\leq 2^s \Vert A \Vert_{\J_s} \left( \sum_{n\in X} \Vert A_{n,l} \Vert + \sum_{n\in X} \Vert A_{k,n} \Vert  \right) .\notag
\end{flalign}
Recall from Proposition \ref{propembedding}, that there exists $C_1 >0$, such that $\Vert A \Vert_{\B} \leq C_1 \Vert A \Vert_{\J_s}$. In particular, for $\theta >0$ (to be chosen later), there exists $\tau_0 >0$, such that 
\begin{equation}\label{est0}
\tau := \Vert A \Vert_{\J_s}^{\theta}\Vert A \Vert_{\B}^{-\theta} \geq C_1^{-\theta} > \tau_0 >0.
\end{equation}
Hence we are in the setting of Lemma \ref{jakoblemma} and may estimate  
\begin{flalign}
&\sum_{n\in X} \Vert A_{n,l} \Vert \notag \\
&\leq \sum_{n\in M_{1,l}^{\tau}} \Vert A_{n,l} \Vert + \sum_{n\in M_{2,l}^{\tau}} \Vert A_{n,l} \Vert \notag\\
&\leq \vert M_{1,l}^{\tau} \vert \Vert A \Vert_{\B} + \Vert A \Vert_{\J_s} \sum_{n\in M_{2,l}^{\tau}} \nu_s(n-l)^{-1} \notag \\
&\leq C_2 \Big( \tau^{d} \Vert A \Vert_{\B} + \Vert A \Vert_{\J_s} \tau^{d-s} \Big) \notag \\
&= C_2 \Big( \Vert A \Vert_{\J_s}^{d\theta} \Vert A\Vert_{\B}^{1-d\theta} + \Vert A \Vert_{\J_s}^{1+(d-s)\theta} \Vert A \Vert_{\B}^{(-(d-s)\theta)} \Big), \notag
\end{flalign}
where we applied Lemma \ref{opnormineq} (b) in the second estimate, and $C_2 >0$ denotes the constant arising in Lemma \ref{jakoblemma}. Analogous reasoning also yields 
\begin{flalign}
&\sum_{n\in X} \Vert A_{k,n} \Vert \notag \\
&\leq C_2 \Big( \Vert A \Vert_{\J_s}^{d\theta} \Vert A\Vert_{\B}^{1-d\theta} + \Vert A \Vert_{\J_s}^{1+(d-s)\theta} \Vert A \Vert_{\B}^{(-(d-s)\theta)} \Big). \notag    
\end{flalign}
Altogether, we obtain
\begin{flalign}
&\Vert [A^2]_{k,l} \Vert \nu_s(k-l) \notag \\
&\leq 2^{s+1} C_2 \Big( \Vert A \Vert_{\J_s}^{1+d\theta} \Vert A\Vert_{\B}^{1-d\theta} + \Vert A \Vert_{\J_s}^{2+(d-s)\theta} \Vert A \Vert_{\B}^{(-(d-s)\theta)} \Big) \notag
\end{flalign}
for all $k,l\in X$. Now, we choose $\theta = \frac{1}{s}>0$, which yields $1 + (d-s)\theta = \frac{d}{s} = d\theta$, and therefore 
$$\Vert [A^2]_{k,l} \Vert \nu_s(k-l) \leq 2^{s+2} C_2 \Vert A \Vert_{\J_s}^{2-\gamma} \Vert A\Vert_{\B}^{\gamma}.$$
Taking the supremum over all $k,l\in X$ yields the claim.
\end{proof}

Now we are finally able to prove the main result of this section. The main contribution to its proof is Lemma \ref{lemmagamma}. In fact, having Lemma \ref{lemmagamma} available at our hands, the proof of the subsequent theorem can be established exactly as the proof of \cite[Theorem 5.15]{hol22}. For completeness reason we provide the details.

\begin{theorem}\label{Jaffardinverseclosed}\cite[Cor. 4.5]{koeba25}
For every $s>d$, the Jaffard algebra $({\J}_s, \vert\Vert \, . \, \Vert\vert_{\J_s})$ is inverse-closed in $\B(\ell^2(X;\Hil))$. In particular, $({\J}_s, \vert\Vert \, . \, \Vert\vert_{\J_s})$ is a symmetric Banach algebra whenever $s>d$. 
\end{theorem}

\begin{proof}
Since $\B := \B(\ell^2(X;\Hil))$ is a symmetric Banach *-algebra, we can deduce both the inverse-closedness of $\J_s$ in $\B$ and the symmetry of $\J_s$ from Hulanicki's Lemma \ref{Hulanickilemma}, once we have verified the inequality of spectral radii
\begin{equation}\label{spectralradiiineq}
    r_{\J_s}(A) \leq r_{\B}(A) \qquad (\forall A = A^* \in \J_s). 
\end{equation}
We establish the verification of (\ref{spectralradiiineq}) via \emph{Brandenburg's trick} \cite{BRANDENBURG75}, which relies on an estimate of the kind (\ref{gammaineq}).

By norm equivalence (\ref{Jaffardnorm}) there exist positive constants $K_1, K_2$ such that 
$$K_1 \vert\Vert A^n \Vert\vert_{\J_s} \leq \Vert A^n \Vert_{\J_s} \leq K_2 \vert\Vert A^n \Vert\vert_{\J_s}$$
for all $A\in \J_s$ and all $n\in \mathbb{N}$. Taking $n$-th roots and letting $n\rightarrow \infty$ implies via Gelfand's formula that  
$$r_{\J_s}(A) = \lim_{n\rightarrow \infty} \vert\Vert A^n \Vert\vert_{\J_s}^{\frac{1}{n}} = \lim_{n\rightarrow \infty} \Vert A^n \Vert_{\J_s}^{\frac{1}{n}} \qquad (\forall A\in \J_s).$$
Now we combine the latter observation with Lemma \ref{lemmagamma} and obtain that
\begin{flalign}
r_{\J_s}(A) &= \lim_{n\rightarrow \infty} \Vert A^{2n} \Vert_{\J_s}^{\frac{1}{2n}} \notag \\
&\leq \lim_{n\rightarrow \infty} C^{\frac{1}{2n}} \Big( \Vert A^{n} \Vert_{\J_s}^{\frac{1}{n}}\Big)^{\frac{2-\gamma}{2}}\Big( \Vert A^{n} \Vert_{\B}^{\frac{1}{n}}\Big)^{\frac{\gamma}{2}} \notag \\
&= r_{\J_s}(A)^{\frac{2-\gamma}{2}}r_{\B}(A)^{\frac{\gamma}{2}} \notag
\end{flalign}
holds for all $A\in \J_s$. Rearranging the latter yields 
$$r_{\J_s}(A)^{\frac{\gamma}{2}} \leq r_{\B}(A)^{\frac{\gamma}{2}} \qquad (\forall A\in \J_s).$$
Since $\gamma>0$, this implies (\ref{spectralradiiineq}). 
\end{proof}

\section*{Acknowledgment}

This work is supported by the project P 34624 \emph{"Localized, Fusion and Tensors of Frames"} (LoFT) of the Austrian Science Fund (FWF).

\bibliographystyle{abbrv}
\bibliography{biblioall}

\begin{thebibliography}{10}

\bibitem{Baskakov1990}
A.~G. Baskakov.
\newblock Wiener's theorem and the asymptotic estimates of the elements of inverse matrices.
\newblock {\em Functional Analysis and Its Applications}, 24(3):222--224, July 1990.

\bibitem{Bas97}
A.~G. Baskakov.
\newblock Estimates for the entries of inverse matrices and the spectral analysis of linear operators.
\newblock {\em Izvestiya: Mathematics}, 61(6):1113, dec 1997.

\bibitem{baskri14}
A.~G. Baskakov and I.~A. Krishtal.
\newblock Memory estimation of inverse operators.
\newblock {\em Journal of Functional Analysis}, 267(8):2551--2605, 2014.

\bibitem{zbMATH06949773}
A.~G. Baskakov and I.~A. Krishtal.
\newblock Spectral properties of an operator polynomial with coefficients in a {Banach} algebra.
\newblock In {\em Frames and harmonic analysis. AMS special session on frames, wavelets and Gabor systems and special session on frames, harmonic analysis, and operator theory, North Dakota State University, Fargo, ND, USA, April 16--17, 2016. Proceedings}, pages 93--114. Providence, RI: American Mathematical Society (AMS), 2018.

\bibitem{BRANDENBURG75}
L.~Brandenburg.
\newblock On identifying the maximal ideals in {B}anach algebras.
\newblock {\em Journal of Mathematical Analysis and Applications}, 50(3):489--510, 1975.

\bibitem{ole1n}
O.~Christensen.
\newblock {\em {A}n {I}ntroduction to {F}rames and {R}iesz {B}ases}.
\newblock Birkh{\"a}user, 2016.

\bibitem{Gosson21}
M.~A. de~Gosson.
\newblock {\em Quantum Harmonic Analysis. An Introduction.}
\newblock De Gruyter, Berlin, Boston, 2021.

\bibitem{douglas}
R.~G. Douglas.
\newblock {\em Banach Algebra Techniques in Operator Theory}.
\newblock Springer, 1998.

\bibitem{feiko98}
H.~G. Feichtinger and W.~Kozek.
\newblock {\em Quantization of TF lattice-invariant operators on elementary LCA groups}, pages 233--266.
\newblock Birkh{\"a}user Boston, Boston, MA, 1998.

\bibitem{gr04-1}
K.~{G}r{\"o}chenig.
\newblock {L}ocalization of {F}rames, {B}anach {F}rames, and the {I}nvertibility of the {F}rame {O}perator.
\newblock {\em {J}. {F}ourier {A}nal. {A}ppl.}, 10(2):105--132, 2004.

\bibitem{gr10-2}
K.~{G}r{\"o}chenig.
\newblock {\em {W}iener's lemma: {T}heme and variations. {A}n introduction to spectral invariance and its applications.}, chapter~5, pages 175 -- 234.
\newblock {A}pplied and {N}umerical {H}armonic {A}nalysis. {B}irkh{\"a}user, 2010.

\bibitem{groelei06}
K.~Gröchenig and M.~Leinert.
\newblock Symmetry and inverse-closedness of matrix algebras and functional calculus for infinite matrices.
\newblock {\em Transactions of the American Mathematical Society}, 358(6):2695--2711, 2006.

\bibitem{hol22}
J.~Holb\"ock.
\newblock Localized frames and applications.
\newblock Master's thesis, University of Vienna, 2022.

\bibitem{Hulanicki1972}
A.~{H}ulanicki.
\newblock {O}n the spectrum of convolution operators on groups with polynomial growth.
\newblock {\em {I}nvent. {M}ath.}, 17:135--142, 1972.

\bibitem{HyNeVeWe16}
T.~Hytönen, J.~van Neerven, M.~Veraar, and L.~Weis.
\newblock {\em Analysis in Banach Spaces, Volume I: Martingales and Littlewood-Paley Theory}.
\newblock Springer, 12 2016.

\bibitem{ja90}
S.~Jaffard.
\newblock Propriétés des matrices “bien localisées” preé de leur diagonale et qualques applications.
\newblock {\em Ann. Inst. H. Poincaré Anal. Non Linéaire}, 7(5):461--476, 1990.

\bibitem{kri11}
I.~Krishtal.
\newblock Wiener's lemma: Pictures at an exhibition.
\newblock {\em Revista de la Unión Matemática Argentina}, 52(2):61--79, 01 2011.

\bibitem{Krishtal2011}
I.~A. Krishtal.
\newblock Wiener's lemma and memory localization.
\newblock {\em Journal of Fourier Analysis and Applications}, 17(4):674--690, August 2011.

\bibitem{kohl21}
L.~Köhldorfer.
\newblock Fusion {F}rames and {O}perators.
\newblock Master's thesis, University of Vienna, 2021.

\bibitem{koeba25-2}
L.~Köhldorfer and P.~Balazs.
\newblock Intrinsically localized g-frames.
\newblock {\em In preparation}, 2025.

\bibitem{koeba25}
L.~Köhldorfer and P.~Balazs.
\newblock Wiener pairs of {B}anach algebras of operator-valued matrices.
\newblock {\em Submitted}, 2025.

\bibitem{koebacasheihomosha23}
L.~Köhldorfer, P.~Balazs, P.~Casazza, S.~Heineken, C.~Hollomey, P.~Morillas, and M.~Shamsabadhi.
\newblock {\em A Survey of Fusion Frames in Hilbert Spaces}, chapter~21.
\newblock Springer Nature Switzerland, 2023.

\bibitem{sun07a}
Q.~{S}un.
\newblock {W}iener's lemma for infinite matrices.
\newblock {\em {T}rans. {A}mer. {M}ath. {S}oc.}, 359(7):3099--3123, 2007.

\bibitem{werner84}
R.~Werner.
\newblock {Quantum harmonic analysis on phase space}.
\newblock {\em Journal of Mathematical Physics}, 25(5):1404--1411, 05 1984.

\bibitem{Wiener32}
N.~Wiener.
\newblock Tauberian theorems.
\newblock {\em Annals of Mathematics}, 33(1):1--100, 1932.

\end{thebibliography}

\end{document}